\documentclass[11pt]{amsart}

\usepackage{amscd,amssymb,amsopn,amsmath,amsthm,mathrsfs,graphics,amsfonts,enumerate,verbatim,calc
}

\usepackage[OT2,OT1]{fontenc}
\newcommand\cyr{%
\renewcommand\rmdefault{wncyr}%
\renewcommand\sfdefault{wncyss}%
\renewcommand\encodingdefault{OT2}%
\normalfont
\selectfont}
\DeclareTextFontCommand{\textcyr}{\cyr} 

\usepackage{amssymb,amsmath}

\DeclareFontFamily{OT1}{rsfs}{}
\DeclareFontShape{OT1}{rsfs}{n}{it}{<-> rsfs10}{}
\DeclareMathAlphabet{\mathscr}{OT1}{rsfs}{n}{it}

\numberwithin{equation}{section}
\newtheorem{theorem}{Theorem}[section]
\newtheorem{lemma}[theorem]{Lemma}
\newtheorem{proposition}[theorem]{Proposition}
\newtheorem{corollary}[theorem]{Corollary}

\newtheorem{question}{Question}
\newtheorem{Problem}{Problem}

\newtheorem*{maintheorem}{Main Theorem}

\theoremstyle{definition}
\newtheorem*{ack}{Acknowledgement}
\newtheorem{definition}[theorem]{Definition}
\newtheorem{remark}[theorem]{Remark}
\newtheorem{example}[theorem]{Example}

\theoremstyle{remark}

\newcommand{\Ht}{\operatorname{ht}}

\newcommand{\Frac}{\operatorname{Frac}}

\newcommand{\perf}{\operatorname{perf}}
\newcommand{\sh}{\operatorname{sh}}

\newcommand{\fm}{\frak{m}}

\newcommand{\fp}{\frak{p}}

\newcommand{\Frob}{\operatorname{Frob}}

\begin{document}
\title[On the Witt vectors of perfect rings in positive characteristic]
{On the Witt vectors of perfect rings in positive characteristic}

\author[K.Shimomoto]{Kazuma Shimomoto}
\address{Department of Mathematics, School of Science and Technology Meiji Univesity 1-1-1 Higashimita Tama-Ku Kawasaki 214-8571 Japan}
\email{shimomotokazuma@gmail.com}
\thanks{The author is partially supported by Grant-in-Aid for Young Scientists (B) \# 25800028}




\begin{abstract}
The purpose of this article is to prove some results on the Witt vectors of perfect $\mathbf{F}_p$-algebras. Let $A$ be a perfect $\mathbf{F}_p$-algebra for a prime integer $p$ and assume that $A$ has the property $\mathbf{P}$. Then does the ring of Witt vectors of $A$ also have $\mathbf{P}$? A main theorem gives an affirmative answer for $\mathbf{P}="\mbox{integrally closed}"$ under a very mild condition.
\end{abstract}

\maketitle

\section{Introduction}

The interplay between rings of prime characteristic $p>0$ and rings of mixed characteristic $p>0$ is the main focus in this article. In a certain situation, these objects are related to each other by the theory of Witt vectors. However, this is fully achieved only when the Frobenius map is an isomorphism. An $\mathbf{F}_p$-algebra $A$ is called \textit{perfect}, if the Frobenius endomorphism on $A$ is bijective. The main theme of this article is to consider the following problem:

\begin{Problem}
\label{mainproblem}
Assume that $A$ is a perfect $\mathbf{F}_p$-algebra and has the property $\mathbf{P}$. Let $\mathbf{W}(A)$ denote the ring of Witt vectors of $A$. Then does $\mathbf{W}(A)$ have also $\mathbf{P}$?
\end{Problem}

To the best of author's knowledge, almost nothing is known about this problem, except for the case when $A$ is a perfect field. Considering recent progress on $p$-adic Hodge theory, the author believes that Problem \ref{mainproblem} will gain interest in light of the importance of the use of Witt vectors. We will investigate the case $\mathbf{P}="\mbox{integrally closed}"$ (see Example \ref{deformnormal} for relevant results). Rings that are treated in this article are usually not Noetherian, since if the Frobenius map on a ring is surjective, it is not Noetherian, unless it is a field. The theory of Witt vectors was invented by Witt himself and has been widely used in number theory in the case when $A$ is a perfect field. However, the construction of the Witt vectors makes sense for any commutative ring, but it is rarely used in the research of commutative algebra. So we hope that this article will shed light on the structure of the Witt vectors for general perfect $\mathbf{F}_p$-algebras. The most important fact to keep in mind is the following. If $A$ is a perfect $\mathbf{F}_p$-algebra, then $\mathbf{W}(A)$ is a $p$-torsion free and $p$-adically complete and separated algebra with an isomorphism: 
$$
\mathbf{W}(A)/p \cdot \mathbf{W}(A) \cong A.
$$
Conversely, the ring $\mathbf{W}(A)$ can be characterized by these properties (see Proposition \ref{prop1}). The most typical case is illustrated by the fact that the ring of Witt vectors of a perfect field of characteristic $p>0$ is a complete discrete valuation ring in which $p$ generates the maximal ideal. We state our main result (see Theorem \ref{normal} and Corollary \ref{completeintegral}):

\begin{maintheorem}
Assume that $R$ is a Noetherian normal domain of characteristic $p>0$ and let $B$ be a perfect integrally closed $\mathbf{F}_p$-domain that is torsion free and integral over $R$. Then the ring of Witt vectors $\mathbf{W}(B)$ is an integrally closed domain.
\end{maintheorem}

\section{Preliminaries}

All rings in this article are assumed to be commutative with unity. However, we are mostly concerned about non-Noetherian rings. The main theme of the present article is the ring of Witt vectors, which is also called the $p$-typical Witt vectors in the literature. Although the ring of Witt vectors is defined for any commutative ring, we only consider them for rings of prime characteristic $p>0$. The basic reference is Serre's book \cite{Se}. Another good source is an expository paper \cite{Rab}. Let $A$ be a commutative ring of characteristic $p>0$. We denote by $\mathbf{W}(A)$ the ring of Witt vectors. Then we have $\mathbf{W}(A)=A^{\mathbf{N}}$ as sets, but its ring structure is defined by complicated operations using Witt-polynomials. We give a brief review of Witt vectors in the next section to the extent we need to prove our main results.

\section{$p$-adic deformation of rings and Witt vectors}

In this section, we introduce a notion of $p$-adic deformation of rings of characteristic $p>0$. First off, we recall the definition of rings with mixed characteristic.

\begin{definition}
Let $p>0$ be a prime number. Say that a ring $A$ has \textit{mixed characteristic $p>0$, or p-torsion free}, if $p$ is a nonzero divisor in $A$ and $pA \ne A$.
\end{definition}

\begin{definition}
An $\mathbf{F}_p$-algebra $B$ is \textit{perfect}, if the Frobenius map $\Frob:B \to B$ is bijective. 
\end{definition}

Now let us consider the following question: Fix a ring $B$ of prime characteristic $p>0$. Then can one find a ring $A$ of mixed characteristic $p>0$ such that $A/pA \cong B$? This naturally leads to the following definition:

\begin{definition}
Let $B$ be a ring of characteristic $p>0$. Then we say that $A$ is a \textit{$p$-adic deformation} of $B$, if the following conditions hold:

\begin{enumerate}
\item[$\bullet$]
$A$ is a ring of mixed characteristic $p>0$ and $A/pA \cong B$.

\item[$\bullet$]
$A$ is $p$-adically complete and separated.
\end{enumerate}
\end{definition}

Let us collect some basic properties of Witt vectors which we often use (see \cite{Se} for more details).

\begin{enumerate}
\item[$\bullet$]
If $A$ is a perfect $\mathbf{F}_p$-algebra, then $\mathbf{W}(A)$ is a $p$-adically complete, $p$-torsion free algebra, and there is a surjection $\pi_A:\mathbf{W}(A) \twoheadrightarrow A$ with kernel generated by $p$.

\item[$\bullet$]
If $A$ is a perfect $\mathbf{F}_p$-algebra, then there is a multiplicative injective map $[-]:A \to \mathbf{W}(A)$ such that the composite map
$$
A \xrightarrow{[-]} \mathbf{W}(A) \xrightarrow{\pi_A} A
$$
is an identity map. $[-]$ is called the \textit{Teichm\"uller mapping}. For any given $x \in \mathbf{W}(A)$, there exist a unique sequence of elements $a_0,a_1,a_2,\ldots \in A$ such that $x=[a_0]+[a_1] p+[a_2] p^2+\cdots$, which we often call the \textit{Witt representation} of $x$. We note the following simple fact.
$$
p|x \iff a_0=0.
$$

\item[$\bullet$]
If $A \to B$ is a ring homomorphism of perfect $\mathbf{F}_p$-algebras, then there is a unique ring homomorphism $\mathbf{W}(A) \to \mathbf{W}(B)$
making the following commutative square:
$$
\begin{CD}
\mathbf{W}(A) @>>> \mathbf{W}(B) \\
@V\pi_AVV @V\pi_BVV \\
A @>>> B \\
\end{CD}
$$
\end{enumerate}

We have a unique $p$-adic deformation for a perfect $\mathbf{F}_p$-algebra, as stated in the following proposition.

\begin{proposition}
\label{prop1}
Let $A$ be a perfect $\mathbf{F}_p$-algebra. Then $A$ admits a unique $p$-adic deformation $\mathcal{A}$. Moreover, if $A \to B$ is a ring homomorphism of perfect $\mathbf{F}_p$-algebras, there exists a unique commutative diagram:
$$
\begin{CD}
\mathcal{A} @>>> \mathcal{B} \\
@VVV @VVV \\
A @>>> B
\end{CD}
$$
such that both $\mathcal{A} \to A$ and $\mathcal{B} \to B$ are $p$-adic deformations.
\end{proposition}

\begin{proof}
The $p$-adic deformation is given as the ring of Witt vectors.  For its uniqueness, the proof is found in (\cite{Se}; Proposition 10 in Chapter II \S 5). To prove the proposition, one can avoid the use of Witt vectors. Alternatively, one uses the cotangent complex (see \cite{Sch}; Theorem 5.11 and Theorem 5.12).
\end{proof}

\begin{definition}
Let $A$ be perfect $\mathbf{F}_p$-algebra. The unique ring map on $\mathbf{W}(A)$ which lifts the Frobenius map on $A$ is called the \textit{Witt-Frobenius map}.
\end{definition}

The following example shows that Proposition \ref{prop1} fails for non-perfect algebras.

\begin{example}
Put $R_1=\mathbf{Z}_p[[x,y]]/(xy)$ and $R_2=\mathbf{Z}_p[[x,y]]/(p-xy)$. Then $R_1$ is non-regular, while $R_2$ is regular, so $R_1$ is not isomorphic to $R_2$. However, we will get
$$
R_1/pR_1 \cong R_2/pR_2 \cong \mathbf{F}_p[[x,y]]/(xy).
$$
The totality of the set of deforming a fixed $\mathbf{F}_p$-algebra to $\mathbf{Z}_p$-adically complete and flat algebras is related to the cotangent complex, and we will explain this a bit in the last section of the article.
\end{example}

\begin{lemma}
\label{p-adic}
Let $B$ be a perfect $\mathbf{F}_p$-domain $($resp. a perfect $\mathbf{F}_p$-algebra$)$. Then the $p$-adic deformation of $B$ is an integral domain $($resp. reduced$)$.
\end{lemma}

\begin{proof}
Assume that $B$ is a perfect $\mathbf{F}_p$-domain. Then $\Frac(B)$ is a perfect field and $\mathbf{W}(\Frac(B))$ is a complete discrete valuation ring. On the other hand, the injection $B \hookrightarrow \Frac(B)$ defines an injection $\mathbf{W}(B) \hookrightarrow \mathbf{W}(\Frac(B))$. Hence $\mathbf{W}(B)$ is an integral domain.

Next, assume that $B$ is a perfect $\mathbf{F}_p$-algebra (hence reduced). Then since $\mathbf{W}(B)$ is $p$-adically complete, it is $p$-adically separated. Using this fact, it is easy to see that $\mathbf{W}(B)$ is reduced.
\end{proof}

We also note the following fact.

\begin{lemma}
\label{adic}
Let $A$ be a perfect $\mathbf{F}_p$-algebra. Then the localization of $A$ and the completion of $A$ with respect to any ideal of $A$ are again perfect.
\end{lemma}

\begin{proof}
This is found in (\cite{Die}; Lemma 3.5 and Lemma 3.6).
\end{proof}

\section{Completion of rings}

For an ideal $I$ of a ring $A$ and an $A$-module $M$, let $M^I:=\varprojlim_{n} M/I^n M$ denote the $I$-adic completion of $M$. For $I=(x)$ a principal ideal, the symbol $M^{\wedge}$ denotes the $x$-adic completion of $M$. Let us recall the following result.

\begin{proposition}[Bartijn-Strooker]
\label{Strooker}
Assume that $A$ is Noetherian. Then $M/IM \cong M^I/IM^I$ and the $A$-module $M^I$ has a natural $A^I$-module structure. Finally, the canonical homomorphism:
$$
M^I \to (M^I)^I
$$
is a bijection, that is, $M^I$ is $I$-adically complete and separated.
\end{proposition}

\begin{proof}
This is found in (\cite{Str}; Theorem 2.2.5, or \cite{Yek}; Corollary 3.5).
\end{proof}

The following lemma is a starting point for embedding a ring into the Witt vectors.

\begin{lemma}
\label{completion}
Let $A$ be a ring. Then the following statement hold.

\begin{enumerate}
\item[$\mathrm{(i)}$]
Let $x$ be a regular element in $A$. Then $x$ is regular in $A^{\wedge}$ $($the $x$-adic completion of $A$$)$.

\item[$\mathrm{(ii)}$]
Let $p$ be a prime integer and assume that $A$ is a $p$-torsion free ring and $A/pA$ is a perfect $\mathbf{F}_p$-algebra. Then there exists a canonical homomorphism:
$$
A \to A^{\wedge} \cong \mathbf{W}(A/pA),
$$
which lifts the identity on $A/pA$. If moreover $A$ is $p$-adically separated, the above map is injective.

\item[$\mathrm{(iii)}$]
Suppose that $A$ is $x$-adically complete and $A/xA$ is $J$-adically complete for an ideal $J$ of $A$. Suppose further that either one of the following conditions holds:

\begin{enumerate}
\item[$\bullet$]
$A$ is a Noetherian ring.

\item[$\bullet$]
$J=(x_1,\ldots,x_d)$ and $(x,x_1,\ldots,x_d)$ is a regular sequence in $A$.
\end{enumerate}
Then $A$ is complete in the $(xA+J)$-adic topology.
\end{enumerate}
\end{lemma}

\begin{proof}
(i): We fix an element
$$
a=(\overline{a}_1,\overline{a}_2,\ldots) \in A^{\wedge} \subset \prod_{n>0} A/x^nA. 
$$
Assume that we have $xa \in x^NA^{\wedge}$ for $N>0$. Since $x\overline{a}_k \in x^N(A/x^kA) \cong (x^NA+x^kA)/x^kA$, we may lift $\overline{a}_k$ to $a_k \in A$ so that $xa_k=x^Nb_k+x^kc_k$ for some $b_k, c_k \in A$. Since $A$ is $x$-torsion free, we have $a_k=x^{N-1}b_k+x^{k-1}c_k \in x^{N-1}A+x^{k-1}A$, which implies that $\overline{a}_{k-1} \in x^{N-1}(A/x^{k-1}A)$, because $\overline{a}_k \in A/x^kA$ maps to $\overline{a}_{k-1} \in A/x^{k-1}A$. Since this holds for any $k>0$, we get
$$
a \in x^{N-1}A^{\wedge}.
$$
Now assume that $xa=0$. Then it is obvious that $xa \in x^NA^{\wedge}$ for all $N>0$, so the above argument yields that
$$
a \in \bigcap_{N>0} x^NA^{\wedge}=0,
$$
as desired.

(ii): Applying (i) to the case $x=p$, together with the fact that $A/pA$ is a perfect $\mathbf{F}_p$-algebra, we obtain an isomorphism $A^{\wedge} \cong \mathbf{W}(A/pA)$ by Proposition \ref{prop1} and Proposition \ref{Strooker}. Hence it induces a canonical homomorphism:
$$
A \to A^{\wedge} \cong \mathbf{W}(A/pA)
$$
and this is injective when $A$ is $p$-adically separated.

(iii): Since (\cite{Shim2}; Lemma 2.2) is slightly vague, we take this opportunity to prove its corrected version. However, other results as well as the main theorem in \cite{Shim2} will be valid under the modified version.

First, assume that $A$ is Noetherian. In this case, the proof of (\cite{Shim2}; Lemma 2.2) goes through without any change by invoking the Artin-Rees Lemma.

Assume that the second condition holds. Then by letting $A_n:=A/x^nA$, there is a short exact sequence of $A$-modules: $0 \to A_{n-1} \xrightarrow{ \cdot x} A_n \to A_1 \to 0$, because $x$ is $A$-regular by assumption. Let $M^J$ denote the $J$-adic completion for an $A$-module $M$. Then we want to show that the $J$-adic completion induces a short exact sequence:
$$
0 \to A_{n-1}^J \to A_n^J \to A_1^J \to 0 \cdots(*).
$$
Since the topology defined by the decreasing filtration $(x_1,\ldots,x_d) \supset (x_1^2,\ldots,x_d^2) \supset \cdots$ is the same as the $J$-adic topology, we have an isomorphism $M^J \cong \varprojlim_{k \in \mathbf{N}} M/(x_1^k,\ldots,x_d^k)M$. Then by assumption, $(x,x_1^k,\ldots,x_d^k)$ is a regular sequence in $A$ for any $k>0$, the snake lemma provides a short exact sequence:
$$
0 \to A_{n-1}/(x_1^k,\ldots,x_d^k)A_{n-1} \to A_n/(x_1^k,\ldots,x_d^k)A_n \to A_1/(x_1^k,\ldots,x_d^k)A_1 \to 0.
$$
Taking the inverse limit of this sequence with respect to $k$, we obtain $(*)$, as required. Now there is a commutative diagram:
$$
\begin{CD}
0 @>>> A_{n-1} @>>> A_n @>>> A_1 @>>> 0 \\
@. @VVV @VVV @VVV \\
0 @>>> A_{n-1}^J @>>> A_n^J @>>> A_1^J  @>>> 0 \\
\end{CD}
$$
The snake lemma implies that $A_n=A/x^nA$ is $J$-adically complete by inductive hypothesis on $n$. Now we have
$$
\varprojlim_{k \in \mathbf{N}} A/(xA+J)^{k} \cong \varprojlim_{m,n \in \mathbf{N}} A/(x^nA+J^m) \cong \varprojlim_{n \in \mathbf{N}}\big(\varprojlim_{m \in \mathbf{N}} A/(x^nA+J^m)\big) \cong \varprojlim_{n \in \mathbf{N}} A/x^nA \cong A,
$$
proving the lemma.
\end{proof}

\begin{remark}
Completion functor behaves nicely in the category of finitely generated modules over Noetherian rings. In general, taking completion of a ring $A$ with respect to an infinitely generated ideal $I \subset A$ can give rise to a pathology. Namely, $I$-adic completion of $A$ is not necessarily $I$-adically complete. Such an example is worked out in detail in (\cite{Yek}; Example 1.8).
\end{remark}

\section{Witt vectors of perfect algebras and main results}

We first recall the definition of the absolute integral closure and refer to \cite{Ar} for more details and properties.

\begin{definition}
Let $B$ be an integral domain. The \textit{absolute integral closure} of $B$ is defined as the integral closure of $B$ in a fixed algebraic closure of the field of fractions of $B$. Denote by $B^+$ this extension ring of $B$. 
\end{definition}

The homological aspect of absolute integral closure of Noetherian domains is explored extensively in \cite{As}. A \textit{normal ring} is a ring $B$ such that the localization of $B$ at every prime ideal is an integrally closed domain in its field of fractions. In what follows, $A^{\wedge}$ will be the $p$-adic completion of a ring $A$.

\begin{proposition}
\label{WittVector}
Let $B$ be a perfect $\mathbf{F}_p$-domain with field of fractions $\mathcal{K}$. Then there is a natural ring isomorphism:
$$
(\mathbf{W}(B)_{(p)})^{\wedge} \cong \mathbf{W}(\mathcal{K}),
$$ 
where $\mathbf{W}(B)_{(p)}$ is the localization at the prime ideal $(p)$. In particular, $\mathbf{W}(B)_{(p)}$ is an unramified discrete valuation ring of mixed characteristic $p>0$.
\end{proposition}

\begin{proof}
Put $\mathcal{B}=\mathbf{W}(B)$. Then this is an integral domain. We first prove that $\mathcal{B}_{(p)}$ is Noetherian. Let $I \subseteq \mathcal{B}_{(p)}$ be an ideal and show that $I$ is finitely generated. Let $\{x_{\lambda}\}$ be a set of generators of $I$, all of which belong to $\mathcal{B}$. Since every element $x \in \mathcal{B}$ is uniquely represented as
$$
x=\sum_{n=0}^{\infty}[a_i]p^i
$$
for the Teichm\"uller lift $[a_i] \in \mathcal{B}$, it follows that 
$$
p^i | x \iff a_0=\cdots a_{i-1}=0.
$$
Moreover, $x$ is a unit of $\mathcal{B}_{(p)}$ if and only if $a_0 \ne 0$.
Let $t$ be the smallest choice of $i$ for which $a_i \ne 0$ for all $x_{\lambda} \in I \subset \mathcal{B}_{(p)}$. Then we have $I=(p^t) \subset \mathcal{B}_{(p)}$. It follows that $\mathcal{B}_{(p)}$ is a Noetherian local domain, but not a field. The above observation shows that every ideal of $\mathcal{B}_{(p)}$ is principal and hence it is a discrete valuation ring. Since $\mathcal{B}_{(p)}$ has perfect residue field $\mathcal{K}$ and its maximal ideal is generated by $p$, it follows that
$$
(\mathcal{B}_{(p)})^{\wedge} \cong \mathbf{W}(\mathcal{K})
$$
by the theory of Witt vectors for perfect fields.
\end{proof}

Let $A$ be a ring of characteristic $p>0$. The \textit{perfect closure} of $A$, which we denote by $A^{\perf}$, is defined to be the colimit of the directed system:
$$
A \xrightarrow{\Frob} A \xrightarrow{\Frob} A \xrightarrow{\Frob} \cdots.
$$
Note that if $A$ is reduced, then $A^{\perf}=\bigcup_{n>0} A^{p^{-n}}$, where $A^{p^{-n}}$ is a ring extension of $A$ obtained by adjoining $p^n$-th roots of all elements of $A$. By definition, the Frobenius map is bijective on $A^{\perf}$. The following lemma enables us to identify the complete regular local ring of mixed characteristic as a subring of the Witt vectors of the absolute integral closure.

\begin{lemma}
\label{regular}
Let $A=k[[z_2,\ldots,z_d]]$ with perfect residue field $k$ of characteristic $p>0$. Then $A$ admits a unique $p$-adic deformation $R$. Moreover, $R$ is a complete regular local ring and the natural inclusion $A \hookrightarrow A^{\perf}$ lifts to a unique $\mathbf{W}(k)$-algebra homomorphism 
$$
R \to \mathbf{W}(A^{\perf}).
$$
\end{lemma}

\begin{proof}
Indeed, $R$ is a complete regular local ring which is isomorphic to $\mathbf{W}(k)[[x_2,\ldots,x_d]]$ with $x_i:=[z_i]$. By the uniqueness of the existence of $p$-adic deformations for perfect algebras, it follows that $\mathbf{W}(A^{\perf})$ is isomorphic to the $p$-adic completion of the algebra:
$$
\bigcup_{n>0} \mathbf{W}(k)[[x_2^{p^{-n}},\ldots,x_d^{p^{-n}}]],
$$
and we get a $\mathbf{W}(k)$-algebra homomorphism $R \to \mathbf{W}(A^{\perf})$. 
\end{proof}

Let $A \to A^+$ be the absolute integral closure for $A=k[[x_2,\ldots,x_d]]$ for a perfect field $k$ of characteristic $p>0$. Then by Lemma \ref{regular}, there is a unique $\mathbf{W}(k)$-algebra homomorphism: 
$$
R=\mathbf{W}(k)[[x_2,\ldots,x_d]] \to \mathbf{W}(A^{\perf}) \to \mathbf{W}(A^+)
$$
which lifts the ring homomorphism $A \to A^{\perf} \to A^+$. Since $A$ is a complete local domain, all of these algebras are quasi-local (a \textit{quasi-local} ring is a ring with a unique maximal ideal). Now we construct a big ring.

\begin{enumerate}
\item[$\bullet$]
Let 
$$
R_{\infty}:=\bigcup_{n>0}\mathbf{W}(k)[[x_2^{p^{-n}},\ldots,x_d^{p^{-n}}]].
$$
Then $R_{\infty}$ is contained in $R^+$ and the Frobenius map $\Frob:R_{\infty}/pR_{\infty} \to R_{\infty}/pR_{\infty}$ is an isomorphism and $R_{\infty}/pR_{\infty}$ is an integral domain.
\end{enumerate}

Moreover, $R_{\infty}$ is a non-Noetherian quasi-local domain. The extension $R \to R_{\infty}$ is integral and faithfully flat, because $R_{\infty}$ is obtained as the filtered colimit of module-finite extensions of regular local rings and each transition map is obviously flat. We keep to use the notation $[a] \in \mathbf{W}(A)$ for the Teichm\"uller lift of $a \in A$. Let $A^{\sh}$ denote the \textit{strict henselization} of a quasi-local ring $A$. For this, we refer the reader to \cite{Ray}.

\begin{lemma}
\label{BigMac}
Let the notation be as above. Then the $R$-algebra $\mathbf{W}(A^+)$ is a big Cohen-Macaulay algebra. In other words, the sequence $(p,x_2,\ldots,x_d)$ is regular on $\mathbf{W}(A^+)$.
\end{lemma}

\begin{proof}
By the main result of \cite{HH}, $A^+$ is a big Cohen-Macaulay $A$-algebra. Then since $\mathbf{W}(A^+)$ is $p$-torsion free, it follows that $(p,x_2,\ldots,x_d)$ is a regular sequence on $\mathbf{W}(A^+)$.
\end{proof}

\begin{lemma}
\label{separated}
Assume that $B$ is a complete Noetherian local domain, $I$ is its maximal ideal, and let $B \subset C \subset B^+$ be a subalgebra. Then $C$ is $I$-adically separated.
\end{lemma}

\begin{proof}
This is found in (\cite{Ho02}; Fact 2.4).
\end{proof}

\begin{proposition}
\label{WittVector2}
Let $K$ be the field of fractions of $A=k[[z_2,\ldots,z_d]]$ for a perfect field $k$ of characteristic $p>0$ and let $K^+$ be the algebraic closure of $K$.

\begin{enumerate}
\item[$\mathrm{(i)}$]
Let $\mathcal{A}:=(R_{\infty})_{(p)}$ be the localization of $R_{\infty}$ at the height-one prime ideal $(p) \subset R_{\infty}$. Then $\mathcal{A}$ is a discrete valuation ring which injects into $\mathbf{W}(K^+)$ and $\mathbf{W}(K^+)$ is the $p$-adic completion of the strict henselization of the $p$-adic completion of $\mathcal{A}$, that is,
$$
\mathbf{W}(K^+) \cong ((\mathcal{A}^{\wedge})^{\sh})^{\wedge}.
$$

\item[$\mathrm{(ii)}$]
If $I$ is any ideal of $R$, then the $R$-algebra $\mathbf{W}(A^+)$ is $I$-adically separated.

\end{enumerate}
\end{proposition}

\begin{proof}
(i): By the construction, the residue field of $\mathcal{A}=(R_{\infty})_{(p)}$ is the perfect closure of the residue field of $R_{(p)}$, and $\mathcal{A}$ is a discrete valuation ring, because it is the ascending union of discrete valuation rings, each of which has the maximal ideal generated by $p$. Then we have a sequence of ring injections $R_{(p)}/p R_{(p)} \subset \mathcal{A}/p \mathcal{A} \subset K^+$. We also note that $\mathcal{A}/p \mathcal{A}$ is a perfect field and there is an injection:
$$
\mathcal{A}^{\wedge}=\mathbf{W}(\mathcal{A}/p\mathcal{A}) \hookrightarrow \mathbf{W}(K^+),
$$ 
which is not integral, but its reduction modulo $p$ is an algebraic extension. From this, we observe that the integral closure of $\mathcal{A}^{\wedge}$ inside $\mathbf{W}(K^+)$ is the filtered colimit of all $\mathbf{W}(L)$ for finite subextensions $\mathcal{A}/p \mathcal{A} \subset L \subset K^+$. Since $L$ is a perfect field, $\mathbf{W}(L)$ is a complete discrete valuation ring and $\mathcal{A}^{\wedge} \to \mathbf{W}(L)$ is module-finite. Hence
$$
(\mathcal{A}^{\wedge})^{\sh}=\varinjlim_{\mathcal{A}/p\mathcal{A} \subset L \subset K^+} \mathbf{W}(L).
$$
Since $(\mathcal{A}^{\wedge})^{\sh}/p (\mathcal{A}^{\wedge})^{\sh}=\mathbf{W}(K^+)/p \mathbf{W}(K^+)$, we conclude that $\mathbf{W}(K^+) \cong ((\mathcal{A}^{\wedge})^{\sh})^{\wedge}$.

(ii): As previously, let $\mathcal{R}=\mathbf{W}(A^+)$. Since $I$ is contained in the maximal ideal of $R=\mathbf{W}(k)[[x_2,\ldots,x_d]]$, it suffices to show that $\mathcal{R}$ is $(p,x_2,\ldots,x_d)$-adically separated. So we may assume that $I=(p,x_2,\ldots,x_d)$ without loss of generality. Then $(p,x_2,\ldots,x_d)$ is a regular sequence on $\mathcal{R}$ in view of Lemma \ref{BigMac} and $A^+=\mathcal{R}/p\mathcal{R}$ is $I$-adically separated in view of Lemma \ref{separated}. Let $M^{I}$ denote the $I$-adic completion of an $\mathcal{R}$-module $M$. Then $A^+ \to (A^+)^I$ is injective and $(A^+)^I$ is a perfect $\mathbf{F}_p$-algebra in view of Lemma \ref{adic}. Then we obtain an injection:
$$
\mathbf{W}(A^+) \hookrightarrow \mathbf{W}((A^+)^I).
$$
Applying (\cite{BH}; Theorem 8.5.1), $(p,x_2,\ldots,x_d)$ is a regular sequence on $\mathbf{W}((A^+)^I)$. Then it follows from Lemma \ref{completion} that $\mathbf{W}((A^+)^I)$ is complete in the $I$-adic topology. In particular, it is separated in the $I$-adic topology. Therefore, $\mathbf{W}(A^+)$ is $I$-adically separated, as desired.
\end{proof}

\subsection{Normality of the Witt vectors}

Now we begin with the next question. Namely, if $B$ is a perfect integrally closed $\mathbf{F}_p$-domain, then is the ring of Witt vectors $\mathbf{W}(B)$ also integrally closed? This is a special case of a general \textit{deformation problem} which is stated as follows. Let $R$ be a ring and $y \in R$ a regular element. Assume that $R/yR$ has some property $\mathbf{P}$. Then does $R$ have $\mathbf{P}$? In the Noetherian local case, Serre's criterion allows one to deform normality, while the deformation of normality fails in the non-Noetherian case (see Example \ref{deformnormal}). We show that the normality condition is preserved under the $p$-adic deformation for certain perfect $\mathbf{F}_p$-algebras. Let us recall the definition of \textit{complete integral closure}.

\begin{definition}
Let $A$ be an integral domain with $K$ its field of fractions. Then an element $x \in K$ is \textit{almost integral over R}, if there is a nonzero element $a \in A$ such that $ax^n \in A$ for all $n \ge 1$. We say that $A$ is \textit{completely integrally closed}, if every element in $K$ which is almost integral over $A$ belongs to $A$.
\end{definition}

In fact, the set of all elements of $K$ that are almost integral over $A$ forms a subring $B$ with $A \subset B \subset K$. This we call the \textit{complete integral closure of A}. It is easy to see from the definition that if $x \in K$ is integral over $A$, then it is almost integral over $A$. The converse is true if $A$ is Noetherian. Note that an integrally closed domain is not necessarily completely integrally closed, as the following example shows.

\begin{example}
Let $V$ be a valuation domain. If $\dim V \ge 2$, then $V$ is normal, but not completely integrally closed. It is also known that $V[[x]]$ is not integrally closed, if $\dim V \ge 2$. In contrast to integral closure, the complete integral closure of an integral domain is not necessarily completely integrally closed (see \cite{Hein} for such an example).
\end{example}

The key point to show that the ring of Witt vectors is normal is to show that it is completely integrally closed, as is demonstrated below. The main idea in the proof of the following theorem is taken from (\cite{Bour}; Proposition 14 in \S 1.4).

\begin{theorem}
\label{normal}
Assume that $R$ is a Noetherian normal domain of characteristic $p>0$ and let $B$ be a perfect integrally closed $\mathbf{F}_p$-domain that is torsion free and integral over $R$. Then the ring of Witt vectors $\mathbf{W}(B)$ is an integrally closed domain.
\end{theorem}

\begin{proof}
Recall that $\mathbf{W}(B)$ is an integral domain and any completely integrally closed domain is integrally closed. So it suffices to prove that $\mathbf{W}(B)$ is a completely integrally closed domain. 

First, we prove that $B$ is a completely integrally closed domain. By assumption, we know that 
$$
R \subset B \subset R^+.
$$ 
More generally, we claim that if $A$ is a Noetherian normal domain in any characteristic, then any integrally closed domain $C$ such that $A \subset C \subset A^+$ is a completely integrally closed domain, which we prove below. Note that
$$
A=\bigcap_{\Ht \fp=1} A_{\fp}
$$
and each $A_{\fp}$ is a discrete valuation ring in view of the fact that a Noetherian normal domain is the intersection of its localizations at all height-one primes. Let $\mathcal{L}$ be the field of fractions of $C$ and let $\overline{A_{\fp}}$ be the integral closure of $A_{\fp}$ in $\mathcal{L}$. Then using (\cite{SwHu}; Theorem 2.1.17), we have
$$
C=\bigcap_{\Ht \fp=1} \overline{A_{\fp}}.
$$
Since $\overline{A_{\fp}}$ is the integral closure of a discrete valuation ring, it is a Pr\"ufer domain of Krull dimension 1 and $\overline{A_{\fp}}$ is equal to the intersection of rank-one valuation rings that are obtained as the localizations at all height-one primes of $\overline{A_{\fp}}$. Since the intersection of rank-one valuation rings is completely integrally closed, $\overline{A_{\fp}}$ and thus $C$ are completely integrally closed as well.

Next, we consider the following commutative diagram:
$$
\begin{CD}
B @>[-]_B>> \mathbf{W}(B) \\
@VVV @VVV \\
\mathcal{K} @>[-]_{\mathcal{K}}>> \mathbf{W}(\mathcal{K})
\end{CD}
$$ 
in which $\mathcal{K}$ is the field of fractions of $B$, and $[-]_B$ and $[-]_{\mathcal{K}}$ denote the Teichm\"uller mappings.

Let $f \in \Frac(\mathbf{W}(B))$ be a nonzero element such that $af^n \in \mathbf{W}(B)$ for all $n \in \mathbf{N}$ and some nonzero element $a \in \mathbf{W}(B)$. Since $\mathbf{W}(\mathcal{K})$ is a discrete valuation ring, we have $f \in \mathbf{W}(\mathcal{K})$. Write
$$
a=\sum_{k=0}^{\infty}[b_k]_{B} \cdot p^k~(\mathrm{resp}.~f=\sum_{k=0}^{\infty}[c_k]_{\mathcal{K}} \cdot p^k),
$$
which is the unique Witt representations with $b_k \in B$ (resp. $c_k \in \mathcal{K}$). To derive a contradiction, we may assume $c_i \notin B$ for some $i$ which we choose as the smallest one. Then since $\sum_{k=0}^{i-1}[c_k]_{\mathcal{K}} \cdot p^k \in \mathbf{W}(B)$, the binomial expansion shows that
$$
a(f-\sum_{k=0}^{i-1}[c_k]_{\mathcal{K}} \cdot p^k)^n \in \mathbf{W}(B)
$$
for all $n \in \mathbf{N}$. Next, choose the smallest $j$ for which $b_j \ne 0$. Then the nonzero term of $a(f-\sum_{k=0}^{i-1}[c_k]_{\mathcal{K}} \cdot p^k)^n$ of the lowest degree with respect to $p$ is
$$
[b_j]_B[ c_i^n]_{\mathcal{K}} \cdot p^{ni+j} \in \mathbf{W}(B).
$$
Let $\sum_{k=0}^{\infty} [d_k]_B \cdot p^k$ be the Witt representation of $[b_j]_B[ c_i^n]_{\mathcal{K}} \cdot p^{ni+j}$ as an element of $\mathbf{W}(B)$. Then noting that the restriction of $[-]_{\mathcal{K}}$ to $B$ coincides with $[-]_B$ and $[b_j]_B[ c_i^n]_{\mathcal{K}}=[b_j c_i^n]_{\mathcal{K}}$, both 
$$
\sum_{k=0}^{\infty} [d_k]_B \cdot p^k~\mbox{and}~[b_j]_B[ c_i^n]_{\mathcal{K}} \cdot p^{ni+j}
$$
give the Witt representations of the same element in $\mathbf{W}(\mathcal{K})$. Then by the uniqueness of the Witt representation, it follows that $d_{ni+j}=b_j c_i^n$ and hence $b_j c_i^n \in B$ for all $n \in \mathbf{N}$.

As we have shown that $B$ is completely integrally closed, the above relation gives us $c_i \in B$, which contradicts our hypothesis that $c_i \notin B$. Therefore, we have $f \in \mathbf{W}(B)$ and this finishes the proof.
\end{proof}

\begin{corollary}
\label{completeintegral}
Let $B$ be a perfect completely integrally closed $\mathbf{F}_p$-domain. Then $\mathbf{W}(B)$ is completely integrally closed, hence integrally closed.
\end{corollary}

\begin{proof}
The proof is the same as above.
\end{proof}

\begin{example}
\label{deformnormal}
\begin{enumerate}
\item[(1)]
In the Noetherian case, the following result is known. Let $(R,\fm)$ be a Noetherian local ring with a regular element $y \in \fm$. Then if $R/yR$ is normal, $R$ is normal. This may be seen by Serre's normality criterion. However this is false in the non-Noetherian case. The following simple example was shown to the author by R. Heitmann. Let
$$
K[x^2,x^3,y,x^2/y,x^3/y,\ldots]
$$
be an algebra over a field $K$ and define $R$ to be its localization at the obvious maximal ideal. Then it is a two-dimensional quasi-local domain with its maximal ideal generated by $y$ and thus $R/yR$ is a field. However, since $x \notin R$ and it is integral over $R$, $R$ is not normal.

\item[(2)]
There is another example of a quasi-local ring for which normality fails to deform. Let $V$ be a valuation domain of Krull dimension $\ge 2$. Then $R=V[[x]]$ is not integrally closed, but $R/xR=V$ is integrally closed.
\end{enumerate}
\end{example}

\begin{example}
We have constructed the ring $R_{\infty}$ previously. The ring $R_{\infty}$ is $p$-torsion free and $R_{\infty}/pR_{\infty}$ is a perfect $\mathbf{F}_p$-algebra. Now the $p$-adic completion of $R_{\infty}$, denoted by $R_{\infty}^{\wedge}$, is isomorphic to $\mathbf{W}(R_{\infty}/pR_{\infty})$. It follows from Theorem \ref{normal} that $R_{\infty}^{\wedge}$ is a normal domain.
\end{example}

\begin{remark}
\begin{enumerate}
\item
Any normal ring $A$ is obtained as the filtered colimit of its Noetherian normal subrings. In fact, it is the filtered colimit of subalgebras which are homomorphic images of algebras of finite type over $\mathbf{Z}$. Since they are excellent, their normalizations in their total ring of fractions are Noetherian normal rings contained in $A$ and their union will be identified with $A$.

\item
Unlike normality, the filtered colimit of completely integrally closed domains such that each transition map is torsion free is not necessarily completely integrally closed. In fact, a valuation domain of Krull dimension at least 2 provides such an example, as it can be written as the filtered colimit of Noetherian integrally closed domains.
\end{enumerate}
\end{remark}

\section{Comments and questions}

The Witt vectors have been rarely used systematically in commutative algebra. In arithmetic geometry, it becomes more common to use the Witt vectors of general commutative algebras, especially in the $p$-adic Hodge theory (see \cite{DavKed}, \cite{KedRuo} for example). We believe that this technique is durable and will bear more applications to attack problems in mixed characteristic.

\begin{question}
Let $A$ be a complete Noetherian local ring of prime characteristic $p>0$. Then what can we say about the set of isomorphism classes of deforming the $\mathbf{F}_p$-algebra $A$ to $\mathbf{Z}_p$-torsion free and $p$-adically complete local algebras? 
\end{question}

A nice answer to this question may be given by the theory of cotangent complex (\cite{Sch}; Theorem 5.11 and Theorem 5.12). Here, we do not assume that a $p$-adic deformation of $A$ is Noetherian. However, we have the following fact. Suppose that $R$ is a $p$-adic deformation of a complete Noetherian local ring $A$. Then $R$ is a complete Noetherian local ring in view of Lemma \ref{completion} together with (\cite{Mat}; Theorem 29.4).

\begin{question}
Let $A$ be a coherent perfect $\mathbf{F}_p$-algebra. Then is the Witt vectors $\mathbf{W}(A)$ also coherent?
\end{question}

For coherent rings and modules, we refer the reader to \cite{Glaz}. This question seems a bit subtle, because it is known that a power series ring over a valuation ring of Krull dimension greater than one is not coherent in \cite{JoSm}. So it seems to require more conditions to have a positive answer.

\begin{question}
Let $A$ be a perfect $\mathbf{F}_p$-algebra and let $\mathbf{F}:\mathbf{W}(A) \to \mathbf{W}(A)$ be the Witt-Frobenius map. Then which subalgebras of $\mathbf{W}(A)$ are stable under $\mathbf{F}$?
\end{question}

If $A$ is a perfect $\mathbf{F}_p$-algebra, then it is not necessarily true that the Witt-Frobenius map $\mathbf{F}:\mathbf{W}(A) \to \mathbf{W}(A)$ stabilizes its subalgebra. However, we have the following fact. Let $B \subset \mathbf{W}(A)$ be its subalgebra. If $b \in B$, then we have $\mathbf{F}(b)-b^p \in p\mathbf{W}(A)$ and this implies that $\mathbf{F}(B+p\mathbf{W}(A)) \subset B+p\mathbf{W}(A)$. The structure of the map $\mathbf{F}$ is rather complicated.

Finally, let us mention that a special case of the Hochster's homological conjecture is established in \cite{Shim3}, using the Witt vectors and the author hopes to use the main result of the present paper in the research of the homological conjectures in a future's occasion.

\begin{ack}
I would like to thank L. E. Miller for teaching the author about Witt vectors and thank P. C. Roberts for suggesting to use the complete integral closure for the proof of the main theorem of this article. Finally, the author would like to thank the people for useful comments during a seminar at Meiji University.
\end{ack}

\end{document}